\documentclass[11pt,a4paper]{article}

\usepackage{amssymb,amsmath,amsthm,color,verbatim}

\newtheorem{thm}{Theorem}[section]
\theoremstyle{definition}
\newtheorem{de}[thm]{Definition}
\newtheorem{ex}[thm]{Example}

\newtheorem{co}[thm]{Corollary}
\newtheorem{re}[thm]{Remark}

\newtheorem{pro}[thm]{Proposition}

{

\DeclareMathOperator*{\argmin}{argmin}

\newcommand{\N}{\mathbb{N}}                                            % Symbol of natural numbers
\newcommand{\R}{\mathbb{R}}                                         % Symbol of real numbers
\newcommand{\eps}{\varepsilon}                                         % Symbol of epsilon
\DeclareMathOperator{\CAT}{CAT}                                        % CAT(k) space.
     
\newcommand{\abs}[1]{\left\vert#1\right\vert}                          % Symbol of absolute value

\parindent0pt
\begin{document}

\title{Moduli of regularity and rates of convergence for Fej\'er monotone sequences}

\author{Ulrich Kohlenbach$^{a}$, Genaro L\'{o}pez-Acedo$^{b}$, Adriana Nicolae$^{b,c}$}
\date{}
\maketitle

\begin{center}
{\scriptsize
$^{a}$Department of Mathematics, Technische Universit\" at Darmstadt, Schlossgartenstra\ss{}e 7, 64289 Darmstadt, Germany
\ \\
$^{b}$Department of Mathematical Analysis - IMUS, University of Seville, Sevilla, Spain
\ \\
$^{c}$Department of Mathematics, Babe\c s-Bolyai University, Kog\u alniceanu 1, 400084 Cluj-Napoca, Romania \\
\ \\
E-mail addresses: kohlenbach@mathematik.tu-darmstadt.de (U. Kohlenbach), glopez@us.es (G. L\'{o}pez-Acedo), anicolae@math.ubbcluj.ro (A. Nicolae)
}
\end{center}

\begin{abstract}
In this paper we introduce the concept of modulus of regularity as a tool to analyze the speed of convergence, including the finite termination, for classes of Fej\'{e}r monotone sequences which appear in fixed point theory, monotone operator theory, and convex optimization. This concept allows for a unified approach to several notions such as weak sharp minima, error bounds, metric subregularity, H\"older regularity, etc., as well as to obtain rates of convergence for Picard iterates, the Mann algorithm, the proximal point algorithm and the cyclic algorithm. As a byproduct we obtain a quantitative version of the well-known fact that for a convex lower semi-continuous function the set of  minimizers coincides with the set of zeros of its subdifferential and the set of fixed points of its resolvent.\\


\end{abstract}

%_______________________________________________________________________________________________________________________________________________________________________________________________________________
%

\section{Introduction}

Many problems in applied mathematics can be brought into the following format: 

\begin{center} 
Let $(X,d)$ be a metric space and $F:X\to \overline\R$ be a function: find a zero of $F$, 
\end{center} 
where as usual $\overline \R = \R \cup \{-\infty,\infty\}$. This statement covers many equilibrium, 
fixed point and minimization problems. Numerical methods, e.g. those based on suitable iterative techniques, 
usually yield sequences $(x_n)$ in $X$ of approximate zeros, i.e. 
$|F(x_n)|< 1/n.$ Based on extra assumptions (e.g., the compactness of 
$X,$ the Fej\'er monotonicity of $(x_n)$ and the continuity of $F$) 
one then shows that $(x_n)$ 
converges to an actual zero $z$ of $F.$ 
An obvious question then concerns the speed 
of the convergence of $(x_n)$ towards $z$ and whether there is an effective 
rate of convergence. \\[1mm] For general families of such problems
formulated for a whole class ${\cal F}$ of functions $F$ one largely has the following 
dichotomy:
\begin{enumerate}
\item[(i)]
 if the zero for $F\in {\cal F}$ is unique, then it usually is possible to give an 
explicit effective rate of convergence, 
\item[(ii)]
if ${\cal F}$ contains functions $F$ with many zeros, one 
usually can use the non-uniqueness to define a (computable) function 
$F\in {\cal F}$ for which $(x_n)$ does not 
have a computable rate of convergence.
\end{enumerate} `(i)' e.g. holds for most fixed point 
iterations involving functions $T:X\to X$ which satisfy some form of 
a contractive condition which guarantees the uniqueness of the fixed 
point (and hence of the zero of $F(x):=d(x,Tx)$). The obvious case, of course, 
is the Banach fixed point theorem, but there are also many situations where 
this is highly nontrivial and tools from logic were used (see \cite{Br(JSL)} 
which in turn is based on methods from \cite{KoBook})  
to extract effective rates of convergence for Picard iterates, 
see, e.g., \cite{ABJL} and the references listed in \cite{Br(JSL)}. \\[1mm] `(ii)' is most strikingly 
exemplified in \cite{Neumann}, where it is shown that all the usual iterations 
used to compute fixed points of nonexpansive mappings already fail in general 
to have computable rates of convergence even for simple computable firmly 
nonexpansive mappings $T:[0,1]\to [0,1].$
\\[1mm] Even though sometimes left implicit, the effectivity of iterative 
procedures in the case of unique zeros (or fixed points) rests on the existence of an effective so-called modulus of uniqueness:  
let $(X,d)$ be a metric space, $F : X \to \overline\R$ with $\text{zer}\; F=\{z\}$ 
and $r>0.$

\begin{de}
We say that $\phi : (0,\infty)  \to (0,\infty)$ is a \textit{modulus of uniqueness} for $F$ w.r.t. $\text{zer}\; F$ and $\overline{B}(z,r)$  
if for all $\eps> 0$ and $x \in \overline{B}(z,r)$ we have the following implication 
\[ |F(x)| < \phi(\eps) \; \Rightarrow \; d(x,z) < \eps.\]
\end{de}
Suppose now that $(x_n)$ is a sequence of $(1/n)$-approximate zeros contained in 
$\overline{B}(z,r)$ for some $r>0.$ If $\phi$ is a modulus of uniqueness for $F$ w.r.t. $\text{zer}\; F$ and 
$\overline{B}(z,r),$ then 
\[ \forall k\ge \lceil 1/\phi(\varepsilon)\rceil \, (d(x_k,z)<\varepsilon). \]

The concept of `modulus of uniqueness' (in the case of families of 
compact metric spaces $K_u$ parametrized by elements $u\in P$ in some 
Polish space $P$) can be found in \cite{Ko} and was used there primarily 
in the context of best approximation theory. In particular, 
it was applied to the uniqueness of the best uniform (Chebycheff) 
approximation of $f\in C[0,1]$ by elements $p$ in some Haar subspace of 
$C[0,1]$ (e.g. the subspace $P_n$ of 
algebraic polynomials of degree $\le n$) and of best approximation in 
the mean ($L^1$-approximation) of $f$ by polynomials in $P_n.$ Proof-theoretic metatheorems applied to the 
nonconstructive uniqueness proofs
in these cases guarantee the extractability of explicit 
moduli of uniqueness (of low complexity) depending only on $n,\varepsilon,$ 
a modulus of continuity $\omega$ of $f$ and some bound 
$M\ge \| f\|_{\infty}$ (where the latter can be avoided in the cases at 
hand by applying a shift $\tilde{f}(x):=f(x)-f(0)$). 
This was explicitly carried out in \cite{Ko,Ko93}
for Chebycheff approximation, where the modulus becomes even 
linear (`constant of strong unicity') if as additional 
input a lower bound $0<l\le \text{dist}(f,P_n)$ is given. The $L^1$-case is treated in \cite{KoOl}. We refer to \cite{KoBook} for more details. \\[1mm] 
In this paper, we are concerned with a generalization of the concept of 
`modulus of uniqueness' called `modulus of regularity' which is applicable also in the non-unique case by considering the distance of a point to the set $\text{zer}\; F$ (see Definition \ref{def-mod-reg}). Note that this concept coincides with that of a `modulus of uniqueness' if $\text{zer}\; F$ is a singleton.
\\[1mm] Again, whenever $(x_n)$ is a sequence of $(1/n)$-approximate zeros of $F$ in $\overline{B}(z,r),$ where $z \in \text{zer}\; F$ and $r > 0$,
$x_k$  is $\varepsilon$-close to some zero 
$z_k\in \text{zer}\; F$ for all $k\ge \lceil 1/\phi(\varepsilon)\rceil.$ \\[1mm] 
A condition which converts this into a rate of convergence is that $(x_n)$ is Fej\'er monotone w.r.t. 
$\text{zer}\;F,$ i.e. for all $z\in \text{zer}\;F$ and $n\in\N$
\[ d(x_{n+1},z)\le d(x_n,z). \] In this case we can infer that for all $k,m\ge  \lceil 1/\phi(\varepsilon)\rceil$ 
\[ d(x_k,x_m)< 2\varepsilon. \]  
So if $X$ is complete and $\text{zer}\; F$ is closed, then $(x_k)$ converges with rate  $\lceil 1/\phi(\varepsilon/2)\rceil$ 
to a zero of $F$ (see Theorem \ref{thm-gen-rate}). 
\\[1mm] As discussed above, in general one cannot expect to have an effective rate of convergence in the non-unique case 
and so the existence of an explicit computable modulus $\phi$ of regularity w.r.t.  $\text{zer}\;F$ will rest on very 
specific properties of the individual mapping $F$ (see Remark \ref{fn-noncomp}).
\\[1mm] Nevertheless, noneffectively one always has a modulus of regularity w.r.t.  $\text{zer}\;F$ if $X$ is compact and 
$F$ is continuous with $\text{zer}\; F\not=\emptyset$ (see Proposition \ref{existence-of-modulus}). This strikingly illustrates
the difference between the unique and the non-unique case: a modulus of uniqueness is a uniform version of having a unique 
zero which - e.g. by logical techniques - can be extracted in effective form from a given proof of uniqueness  
\[ F(x)=0=F(z) \to x=z, \]  
(see \cite[Section 15.2]{KoBook} with Corollary 17.54 instead of Theorem 15.1 
to be used in the noncompact 
case) whereas a modulus of regularity w.r.t.  $\text{zer}\;F$ is a uniform version of the trivially true property 
\[ F(x)=0  \to \forall \varepsilon>0\,\exists z\in \text{zer}\; F\, (d(x,z)<\varepsilon), \] which, 
however, has too complicated a logical form to guarantee - for computable $F$ 
and effectively represented $X$ - computability even in the presence of compactness. \\[1mm]
While the concept of a modulus of regularity (and also Proposition \ref{existence-of-modulus}) 
has been used in various special situations before (see, e.g., \cite{Anderson} and the literature cited there), 
we develop it in this paper 
as a general tool towards a unified treatment of a number of concepts studied in convex optimization such as weak sharp minima, error bounds, metric subregularity, H\"older regularity, etc., which can be seen as instances of moduli
of regularity w.r.t. $\text{zer}\;F$ for suitable choices of $F.$ Actually, as it will be pointed out in Section \ref{section-modulus}, for minimization problems the notion of modulus of regularity is tightly related to the ones of weak sharp minima or error bounds. 
\\[1mm] 
After some preliminaries, we show in Section \ref{section-modulus} how 
the concept of `modulus of regularity' w.r.t. $\text{zer}\; F$ can be 
specialized to suitable notions of `modulus of regularity' for 
equilibrium problems, fixed point problems, the problem of finding a zero 
of a set-valued operator and minimization problems. In Theorem 
\ref{thm-equiv-mod}, 
we give - in terms of the respective moduli of regularity - a quantitative version 
of the well-known identities between 
minimizers of proper, convex and lower semi-continuous functions $f,$ 
fixed points of the resolvent $J_{\gamma \partial f}$ of $\partial f$ 
of order $\gamma>0$ and the zeros of $\partial f:$
$$\argmin \; f = \text{Fix} \; J_{\gamma \partial f} = \text{zer}\; \partial f.$$ 
In Section \ref{section-rates} we use the concept of `modulus of regularity' to give a general convergence result, Theorem \ref{thm-gen-rate}, which provides, under suitable assumptions, explicit rates of convergence for Fej\'{e}r monotone sequences. In particular, this result can be employed for various iterative methods such as Picard and Mann iterations, cyclic projections, as well as the 
proximal point algorithm. Together with 
the concept of metric regularity for finite families of intersecting sets 
from \cite{B96}, this also applies to convex feasibility problems. 

\section{Preliminaries} \label{section-prelim}
Throughout this paper, if not stated otherwise, $(X,d)$ stands for a complete metric space, which is the natural setting for the concepts and results contained in this work. Although most of the algorithms considered in the subsequent sections are defined in Hilbert spaces, which we usually denote by $H$, in some situations we also refer to the context of $\CAT(\kappa)$ spaces, $\kappa \in \R$, which are also known as Alexandrov spaces of curvature bounded above by $\kappa$ and which we define in the sequel. 

For $x \in X$ and $r > 0$, we denote the {\it open ball} and the {\it closed ball} centered at $x$ with radius $r$ by $B(x,r)$ and $\overline{B}(x,r)$, respectively. If $C$ is a subset of $X$, the {\it distance} of a point $x \in X$ to $C$ is $\text{dist}(x,C) := \inf\{d(x,c) : c \in C\}$. Having $x,y \in X$, a {\it geodesic} from $x$ to $y$ is a mapping $c:[0,l]\subseteq \R \to X$ such that $c(0)=x$, $c(l)=y$ and $d(c(t),c(s))=\abs{t-s}$ for all $t,s\in[0,l]$. The image of $c$ is called a {\it geodesic segment} joining $x$ to $y$ and is not necessarily unique.  A point $z \in X$ belongs to a geodesic segment joining $x$ to $y$ if and only if there exists $t\in [0,1]$ such that $d(x,z)=td(x,y)$ and $d(y,z)=(1-t)d(x,y)$ and we write $z=(1-t)x+ty$ if no confusion arises. We say that $X$ is a {\it (uniquely) geodesic metric space} if every two points in it are joined by a (unique) geodesic. A set $C$ in a uniquely geodesic metric space is called {\it convex} if given any two points in $C$, the geodesic segment joining them is contained in $C$.

One way to define $\CAT(\kappa)$ spaces is via a quadrilateral condition which we state next for the case $\kappa = 0$. More precisely, a geodesic metric space $(X,d)$ is said to be a {\it $\CAT(0)$ space} if 
\begin{equation}\label{def-CAT0}
d(x,y)^2 + d(u,v)^2 \le d(x,v)^2 + d(y,u)^2 + d(x,u)^2 + d(y,v)^2,
\end{equation}
for any $x,y,u,v \in X$ (see \cite{BerNik08}). Any $\CAT(0)$ space is uniquely geodesic. The Hilbert ball with the hyperbolic metric is a prime example of a CAT$(0)$ space, see \cite{GR84}. Other examples include Hilbert spaces, $\mathbb{R}$-trees, Euclidean buildings, Hadamard manifolds, and many other important spaces. When $\kappa \ne 0$, a related inequality recently given in \cite{BerNik15} can be used to introduce $\CAT(\kappa)$ spaces. 

\mbox{}

In the following we recall definitions and properties of operators which are significant in this paper. We refer to \cite{BC17} for a detailed exposition on this topic. Let $(X,d)$ be a metric space, $C \subseteq X$ nonempty and $T:C\rightarrow X$. The {\it fixed point set} of $T$ is denoted by $\text{Fix}\; T :=\{x\in C : Tx = x\}$. The mapping $T$ is said to be {\it nonexpansive} if $d(T(x), T(y)) \leq d(x,y)$ for all $x,y \in C$. Likewise, $T$ is said to be {\it quasi-nonexpansive} if $\text{Fix}\; T \ne \emptyset$ and $d(T(x), z) \le d(x,z)$ for all $z \in \text{Fix}\; T$. Suppose next that $(X,d)$ is a uniquely geodesic metric space. If $C$ is closed and convex, then $\text{Fix}\; T$ is also closed and convex, whenever $T$ is quasi-nonexpansive. We say that $T$ is {\it firmly nonexpansive} if 
\[d(T(x),T(y)) \le d((1-\lambda)x + \lambda T(x),(1-\lambda)y + \lambda T(y)),\]
for all $x,y\in C$ and $\lambda\in [0,1]$. When $X$ is a Hilbert space, there are several equivalent definitions of firm nonexpansivity, one of them being that $T$ can be written as $T = (1/2)\text{Id} + (1/2)S$, where $S$ is nonexpansive.

Let $A$ be a set-valued operator defined on a Hilbert space $H$, $A: H \rightarrow 2^H$. We say that $A$ is {\it monotone} if $\langle x^*-y^*,x-y\rangle \geq0$ for all $x,y \in H$, $x^* \in A(x)$, $y^* \in A(y)$. Suppose next that $A$ is monotone. The {\it resolvent} of $A$ of order $\gamma > 0$ is the mapping $J_{\gamma A} := (\text{Id} + \gamma A)^{-1}$ defined on $\text{ran}(\text{Id} + \gamma A)$, which can be shown to be single-valued and firmly nonexpansive. Denoting the set of zeros of $A$ by $\text{zer}\; A := \{x \in H : O \in A(x)\}$, we immediately have $\text{Fix}\; J_{\gamma A} = \text{zer}\; A$. The {\it reflected resolvent} is the mapping $R_{\gamma A} := 2J_{\gamma A} - \text{Id}$, which is nonexpansive as $J_{\gamma A}$ is firmly nonexpansive. If the monotone operator $A$ has no proper monotone extension, then it is called {\it maximal monotone}. In this case $J_{\gamma A}$ and $R_{\gamma A}$ are defined on $H$. 

Let $f:H\to (-\infty,\infty]$ be proper. The {\it subdifferential} of $f$ is the set-valued operator $\partial f : H \to 2^H$ defined by
\[\partial f(x):=\{u\in H:\langle u, y-x\rangle\leq f(y)-f(x),\;\forall y\in H\}.\]
It is easy to see that $\partial f$ is monotone. Denoting the {\it set of minimizers} of $f$ by $\argmin f := \{x\in H : f(x)\le f(y),\;\forall y\in H\}$, we have $\text{zer}\;\partial f = \argmin f = \text{Fix} \; J_{\gamma \partial f}$.

Let $C \subseteq H$ be nonempty and convex. Recall that the {\it indicator function} $\delta_C : H \to [0,\infty]$ is defined by
\[
\delta_C(x):=\left\{
\begin{array}{ll}
0, & \mbox{if } x \in C,\\
\infty, & \mbox{otherwise}
\end{array}
\right.
\]
and the {\it normal cone map} $N_C : H \to 2^H$ is
\[
N_C(x):=\left\{
\begin{array}{ll}
\{u \in H : \langle u, c - x \rangle \le 0,\;\forall c\in C\}, & \mbox{if } x \in C,\\
\emptyset, & \mbox{otherwise}.
\end{array}
\right.
\]
Clearly, $\partial \delta_C = N_C$.
 
Suppose now that $f$ is additionally convex and lower semi-continuous. Then $\text{int}\; \text{dom}\; f = \text{cont}\; f \subseteq \text{dom}\;\partial f \subseteq \text{dom}\; f$ and $\partial f $ is a maximal monotone operator. Note that if $C \subseteq H$ is nonempty, closed and convex, then $\delta_C$ is proper, convex and lower semi-continuous and $N_C$ is maximal monotone.

The mapping $\text{Prox}_f : H \to H$,
\begin{equation}\label{eq:def-F-mu}
\text{Prox}_f(x):=\argmin_{y\in H}\left(f(y)+\frac{1}{2}\|x - y\|^2\right),
\end{equation}
is well-defined and called the {\it proximal mapping} of $f$. Note that $J_{\gamma \partial f} = \text{Prox}_{\gamma f}$ for all $\gamma > 0$. One can also show that 
\begin{equation} \label{ineq-res}
f(J_{\gamma \partial f}(x)) - f(y) \le \frac{1}{2\gamma}\left(\|y-x\|^2 - \|J_{\gamma \partial f}(x) - x\|^2 - \|J_{\gamma \partial f}(x) - y\|^2\right),
\end{equation}
for every $\gamma > 0$ and $x,y \in H$ (see, e.g., \cite[Lemma 3.2]{Bac14}).

\mbox{}

The metric projection also plays an important role in our further discussion. Let $(X,d)$ be a metric space and $C \subseteq X$ nonempty. The {\it metric projection} $P_C$ onto $C$ is the mapping $P_C : X\to 2^C$ defined by $P_C(x):=\{ y \in C : d(x,y)=\mbox{dist}(x,C)\}$. If $X$ is a complete $\CAT(0)$ space and $C$ is nonempty, closed and convex, then $P_C : X \to C$ is well-defined, single-valued and firmly nonexpansive. Moreover, 
\begin{equation}\label{eq-CAT0-proj}
d(x,P_C x)^2 + d(P_C x, y)^2 \le d(x,y)^2,
\end{equation}
for any $x \in X$ and $y \in C$. Note that in Hilbert spaces, $J_{N_C} = \text{Prox}_{\delta_C} = P_C$.

\mbox{}

The notions of Fej\'er monotonicity and asymptotic regularity are central in the study of convergence of algorithms associated to nonexpansive-type operators. Let $(X,d)$ be a metric space and $C\subseteq X$ nonempty. A sequence $(x_n) \subseteq X$ is {\it Fej\'er monotone} with respect to $C$ if $d(x_{n+1},p)\le d(x_n,p)$ for all $n\in\N$ and $p\in C$. We say that an iteration $(x_n) \subseteq C$ associated to a mapping $T : C \to C$ is {\it asymptotically regular} if $\lim_{n \to \infty}d(x_n,Tx_n) = 0$ for any starting point in $C$. In this case, a function $\alpha : (0,\infty) \to \N$ is a {\it rate of asymptotic regularity} for $(x_n)$ if
\[\forall \varepsilon > 0 \, \forall n\ge \alpha(\eps) \, \left( d(x_n, Tx_n)<\eps\right).\]

We end this subsection with a definition that will be needed later on. A function $\theta:\N\to\N$ is a {\it rate of divergence} for a series $\sum_{n \ge 0}\gamma_n$ if  $\gamma_n \ge 0$ and $\sum_{k=0}^{\theta(n)} \gamma_k \ge n$ for all $n \in \N$.

\section{Modulus of regularity}\label{section-modulus}
Let $(X,d)$ be a metric space and $F : X \to \overline{\R}$ with $\text{zer}\; F \ne \emptyset$.

\begin{de}\label{def-mod-reg}
Fixing $z \in \text{zer}\; F$ and $r > 0$, we say that $\phi : (0,\infty)  \to (0,\infty)$ is a \textit{modulus of regularity} for $F$ w.r.t. $\text{zer}\; F$ and $\overline{B}(z,r)$ if for all $\eps> 0$ and $x \in \overline{B}(z,r)$ we have the following implication
\[|F(x)| < \phi(\eps) \; \Rightarrow \; {\rm dist}(x, \text{zer}\; F) < \eps.\]

If there exists $z \in \text{zer}\; F$ such that $\phi : (0,\infty)  \to (0,\infty)$ is a modulus of regularity for $F$ w.r.t. $\text{zer}\; F$ and $\overline{B}(z,r)$ for any $r > 0$, then $\phi$ is said to be a \textit{modulus of regularity} for $F$ w.r.t. $\text{zer}\; F$.
\end{de}

Our first result shows that such a modulus always exists when the domain is compact and the function is continuous.
 
\begin{pro}\label{existence-of-modulus}
If $X$ is proper and $F$ is continuous, then for any $z \in \text{zer}\; F$ and $r > 0$, $F$ has a modulus of regularity w.r.t.  $\text{zer}\; F$ and $\overline{B}(z,r)$.
\end{pro}
\begin{proof}  
It is enough to prove that 
\[ \forall \eps >0 \,\exists n\in\N \setminus \{0\} \,\forall x\in \overline{B}(z,r) \left(\left|F(x)\right|< \frac{1}{n} \to \exists q\in \text{zer}\; F \,\left(d(x,q)<\eps\right)\right).\]
Assume that this is not the case. Then there exist $\eps > 0$ and a sequence $(x_n)$ in $\overline{B}(z,r)$ such that 
\begin{equation}\label{prop-existence-of-modulus-eq1}
\forall n\in\N \setminus \{0\}  \, \left(\left|F(x_n )\right|< \frac{1}{n} \wedge \forall q\in \text{zer}\; F  \,\left(d(x_n,q)\ge \eps\right)\right).
\end{equation}
Let $\widehat{x}$ be a limit point of $(x_n).$ Then, using the continuity 
of $F,$ we get $F(\widehat{x}) =0,$ i.e. $\widehat{x}\in \text{zer}\; F .$ 
Also \[ \exists n\in\N \, \left(d(x_n,\widehat{x})<\eps\right).\]
Putting $q:=\widehat{x},$ this contradicts the last conjunct in \eqref{prop-existence-of-modulus-eq1}. 
\end{proof}

\begin{re}\label{rmk-existence-modulus}
From the above, it follows that if $X$ is compact and $F$ is continuous, then $F$ has a modulus of regularity w.r.t.  $\text{zer}\; F$.
\end{re}
\begin{re} The proof of Proposition \ref {existence-of-modulus} is 
noneffective and in general even for simple computable functions $F$ there 
is no computable modulus of regularity (see Remark \ref{fn-noncomp} below). 
A characterization 
of the proof-theoretic strength and the computability-theoretic status 
of Proposition \ref {existence-of-modulus} in terms of `reverse mathematics' 
and Weihrauch complexity is given in \cite{Koh18}.
\end{re}
 
The notion of modulus of regularity  appears in a natural way in different relevant problems such as the following ones.

\subsection*{Equilibrium problems}
Given the nonempty subsets $C$ and $D$ of two Hilbert spaces $H_1$ and $H_2$, respectively, and a mapping $G: C\times D\to \R$, the equilibrium problem associated to the mapping $G$ and the sets $C$ and $D$ consists of finding an element $p \in C$ such that
\begin{equation}\label{def-eq}
G(p,y)\geq  0,
\end{equation}
for all $y \in D$.

Suppose that the set of solutions for problem \eqref{def-eq}, denoted by $\text{EP}(G,C,D)$, is nonempty and define $F : C \to \R$,
$$F(x) :=  \min\left\{0,\inf_{y\in D} G(x,y)\right\}.$$  
Note that $\text{zer}\; F = \text{EP}(G,C,D).$

Let $z \in \text{EP}(G,C,D)$ and $r > 0$. A {\it modulus of regularity} for $G$ w.r.t. $\text{EP}(G,C,D)$ and $\overline{B}(z,r)$ is a modulus of regularity for $F$ w.r.t. $\text{zer}\; F$ and $\overline{B}(z,r)$. This modulus appears, under the name of error bound, in the study of parametric inequality systems. In \cite{LMNP}, such an approach is used to obtain rates of convergence for the cyclic projection method employed in solving convex feasibility problems.

The equilibrium problem covers in particular the classical variational inequality problem. Given a nonempty, closed and convex subset $C$ of a Hilbert space $H$ and a mapping $A: C\to H$, the classical variational inequality problem associated to $A$ and $C$ consists of finding an element $z \in C$ such that
\begin{equation}\label{def-vi}
\langle A(z), y-z\rangle \ge 0,
\end{equation}
for all $y \in C$.
Denote by $\text{VI}(A,C)$ the set of solutions for problem \eqref{def-vi} and assume that it is nonempty. In this case one considers $G : C \times C \to \R$ defined by $G(x,y) := \langle A(x), y-x \rangle$ and, for $z \in \text{VI}(A,C)$ and $\overline{B}(z,r)$, a {\it modulus of regularity} for $A$ w.r.t. $\text{VI}(A,C)$ and $\overline{B}(z,r)$ is a modulus of regularity for $G$ w.r.t. $\text{EP}(G,C,C)$ and $\overline{B}(z,r)$.

\subsection*{Fixed point problems} 
Let $(X,d)$ be a metric space, $T : X \to X$ with $\text{Fix}\; T \ne \emptyset$ and define $F : X \to \R$ by $F(x) := d(x,Tx)$. Note that  $\text{zer}\; F = \text{Fix}\; T$.

Let $z \in \text{Fix}\; T$ and $r > 0$. A {\it modulus of regularity} for $T$ w.r.t. $\text{Fix}\; T$ and $\overline{B}(z,r)$ is a modulus of regularity for $F$ w.r.t. $\text{zer}\; F$ and $\overline{B}(z,r)$. In a similar way, a {\it modulus of regularity} for $T$ w.r.t. $\text{Fix}\; T$ is defined to be a modulus of regularity for $F$ w.r.t. $\text{zer}\; F$.

This concept appears in particular forms in \cite{Borwein17} and \cite{LFP14,LFP16} where it was used, respectively, to study the linear and H\"{o}lder local convergence for algorithms related to nonexpansive mappings. 

\mbox{}

The next result is  a direct consequence of Proposition \ref{existence-of-modulus} and Remark \ref{rmk-existence-modulus}.
\begin{co}\label{existence-of-modulusfp}
If $X$ is proper, $T$ is continuous, $z \in \text{Fix}\;T$ and $r > 0$, then $T$ has a modulus of regularity w.r.t. $\text{Fix}\;T$ and $\overline{B}(z,r)$. If $X$ is additionally compact, then $T$ has a modulus of regularity w.r.t. $\text{Fix}\;T$.
\end{co}

In the following section (see Remark \ref{fn-noncomp}) we show that even simple computable firmly nonexpansive mappings $T:[0,1]\to [0,1]$ may not have a computable modulus of regularity w.r.t. $\text{Fix}\; T$. 

\mbox{}

We give next three concrete instances of moduli of regularity that are computed explicitly.
\begin{ex}\label{ex-mwu-fix}
\begin{enumerate}
\item[(i)]
Let $X$ be a complete metric space and $T : X \to X$ a contraction with constant $k \in [0,1)$. Then $\text{Fix}\; T = \{z\}$ for some $z \in X$ and it is easy to see that $\phi(\varepsilon):=(1-k)\varepsilon$ is a modulus of regularity for $T$ w.r.t. $\text{Fix}\; T$ (in fact it is even a modulus of uniqueness). Indeed, $d(x,Tx) < (1-k)\varepsilon$ yields
\[d(x,z) \le d(x,Tx) + d(Tx,Tz) < (1-k)\eps + kd(x,z),\]
hence $d(x,z) < \eps$.
\item[(ii)] 
Let $X$ be a complete metric space and $T:X \to X$ an orbital contraction with constant $k \in [0,1)$ (i.e. $d(Tx,T^2x) \le kd(x,Tx)$ for all $x \in X$). If $T$ is additionally continuous, one can show that $\phi(\varepsilon):=(1-k)\varepsilon$ is a modulus of regularity for $T$ w.r.t. $\text{Fix}\; T$. To this end let $x \in X$ with $d(x,Tx) < \phi(\eps)$ and let $n,l \in \N$. Then
\[
d(T^n x, T^{n+l}x) \le \sum_{i=0}^{l-1} d(T^{n+i}x,T^{n+i+1}x) \le \sum_{i=0}^{l-1}k^{n+i}d(x,Tx) \le \frac{k^n}{1-k}d(x,Tx).
\]
This shows that $(T^n x)$ is a Cauchy sequence, hence it converges to some $z \in X$. Note that since $T$ is continuous, $z \in \text{Fix}\; T$. Moreover, $d(Tx,z) \le \frac{k}{1-k}d(x,Tx)$ and so
\[\text{dist}(x, \text{Fix}\; T) \le d(x,z) \le d(x,Tx) + d(Tx,z) \le \frac{1}{1-k}d(x,Tx) < \eps.\]

We include below an example of a continuous orbital contraction which has more than one fixed point and refer to \cite{PetRus18} for a more detailed discussion on orbital contractions.

Let $X = \{(x,y) \in \R^2: 0 \le x \le 1, 0 \le y \le 1-x\}$ with the usual Euclidean distance. Define $f : X \to X$ by  
\[f(x,y) = \left(x,\frac{y+1-x}{2}\right).\]
Then $f$ is continuous, $\|f^2(x,y) - f(x,y)\| = \|f(x,y)-(x,y)\|/2$ for all $(x,y) \in X$ and $\text{Fix}\; T = \{(x,1-x):x\in[0,1]\}$.
\item[(iii)]
Let $X$ be a metric space and $T:X\to C\subseteq X$ be a retraction. Then $\phi(\varepsilon):=\varepsilon$ 
is a modulus of regularity for $T$ w.r.t. $\text{Fix}\; T$. To see this, note that 
$\text{Fix}\; T=T(X)=C$ and $d(x,Tx)<\varepsilon$ implies $\text{dist}(x,\text{Fix}\; T)<\varepsilon$ since $Tx\in \text{Fix}\; T$ (not even the continuity of $T$ is needed for this). In particular, this applies to the case where $T$ is the metric projection of $X$ onto $C$ if the metric projection exists as a single-valued function.
\item[(iv)]
For nonempty, closed and convex subsets $C_1,C_2\subseteq \R^n$ consider 
\[ T:=R_{N_{C_2}}R_{N_{C_1}}. \]
In \cite[p. 18]{Borwein17} it is shown that if $C_1,C_2$ are convex 
semi-algebraic sets with $O \in C_1\cap C_2$ which can be described 
by polynomials on $\R^n$ of degree greater than $1$, then (in our terminology), given $r > 0$,
$T$ admits the following modulus of regularity w.r.t. $\text{Fix}\; T$ and $\overline{B}(O, r)$
\[\phi(\varepsilon):=2(\varepsilon/\mu)^{\gamma},\] 
for suitable $\mu>0$ and $\gamma\ge 1$.
\end{enumerate} 
\end{ex}

\subsection*{Minimization problems} 
Let $(X,d)$ metric space and $f : X \to (-\infty,\infty]$. We consider the problem
\begin{equation}\label{def-pb-min}
\argmin_{x\in X }f(x).
\end{equation}
Suppose that its set of solutions $S$ is nonempty and denote $m:=\min_{x\in X }f(x)$. Define the function $F : X \to \overline{\R}$, $F(x) := f(x) - m$. Note that $\text{zer}\; F = S$ and $F(x) = \infty$ for $x \notin \text{dom}\; f$.

Given $z \in S$ and $r > 0$, a {\it modulus of regularity} for $f$ w.r.t. $S$ and $\overline{B}(z,r)$ is a modulus of regularity for $F$ w.r.t. $\text{zer}\; F$ and $\overline{B}(z,r)$. Similarly, a {\it modulus of regularity} for $f$ w.r.t. $S$ is a modulus of regularity for $F$ w.r.t. $\text{zer}\; F$. This concept is closely related to growth conditions for the function $f$ such as the notions of sets of weak sharp minima or error bounds (see, e.g., \cite{BF93,Fer91,BD02,Bol17,LMWY} with the remark that there is a vast literature on these topics and their connection to other regularity properties). These conditions are especially used to analyze the linear convergence or the finite termination of central algorithms in optimization.

\mbox{}

In the following $S$ stands as above for the set of solutions of problem \eqref{def-pb-min}.
\begin{ex} \label{ex-min-pb}
\begin{itemize}
\item[(i)] The set $S$ is called a set of {\it $\psi$-global weak sharp minima} for $f$ if 
\begin{equation}\label{ineq-weak-sharp}
f(x) \ge m + \psi(\text{dist}(x,S)),
\end{equation}
for all $x \in X$, where $\psi :[0,\infty) \to [0,\infty)$ is a strictly increasing function satisfying $\psi(0) =0$. In this case, $\phi : (0,\infty) \to (0,\infty)$, $\phi(\eps):=\psi(\eps)$, acts as a modulus of regularity for $f$ w.r.t. $S$. The case $\psi(\eps) =k\,\eps$ with $k >0$ was introduced in \cite{BF93}. 

\item[(ii)] More generally, one can assume that $S$ is a set of {\it $\psi$-boundedly weak sharp minima} for $f$, that is, for any bounded set $C \subseteq X$ with $C \cap S \ne \emptyset$, there exists a strictly increasing function $\psi = \psi_C : [0,\infty) \to [0,\infty)$ satisfying $\psi(0) = 0$ such that \eqref{ineq-weak-sharp} holds for all $x \in C$. Fixing $z \in S$ and $r > 0$, a modulus of regularity for $f$ w.r.t. $S$ and $\overline{B}(z,r)$ can be defined by $\phi : (0,\infty) \to (0,\infty)$, $\phi(\eps) := \psi_C(\eps)$, where $C := \overline{B}(z,r)$. 
\end{itemize}
\end{ex}

\begin{re}
In this regard, if $\omega$ is an increasing function satisfying $\omega(0) = 0$, an inequality of the form 
\[\omega(f(x)-m) \ge \text{dist}(x,S),\] 
where $x$ either lives in $X$ or in a bounded set, is also called an {\it error bound}.
\end{re}

\begin{re}
At the same time, if $\phi$ is a modulus of regularity for $f$ w.r.t. $S$ and $\overline{B}(z,r)$, then $f(x) \ge m + \phi(\text{dist}(x,S))$, for all $x \in \overline{B}(z,r)$. Indeed, supposing that there exists $x \in \overline{B}(z,r)$ such that $f(x) -m < \phi(\text{dist}(x,S))$, then $\text{dist}(x,S) < \text{dist}(x,S)$, a contradiction. Thus, a modulus of regularity also induces a growth condition for the function $f$.
\end{re}

\subsection*{Zeros of set-valued operators}
Let $X$ and $Y$ be normed spaces and $A : X \rightarrow 2^Y$ be a set-valued operator such that $\text{zer}\; A \ne \emptyset$ and
\begin{equation}\label{eq-cond-oper}
{\rm dist}(O_Y, A(x)) = 0 \; \Rightarrow \; x \in \text{zer}\; A,
\end{equation}
for all $x \in X$. If $F : X \to \overline{\R}$ is defined by $F(x) := \text{dist}(O_Y,A(x))$, then $\text{zer}\; F = \text{zer}\; A$ and $F(x) = \infty$ for $x \notin \text{dom}\; A$. Note that if $H$ is a Hilbert space and $A : H \rightarrow 2^H$ is maximal monotone, then $A(x)$ is closed for all $x \in H$, so \eqref{eq-cond-oper} holds.

Given $z \in \text{zer}\; A$ and $r > 0$, a {\it modulus of regularity} for $A$ w.r.t. $\text{zer}\; A$ and $\overline{B}(z,r)$ is a modulus of regularity for $F$ w.r.t. $\text{zer}\; F$ and $\overline{B}(z,r)$.  Similarly, a {\it modulus of regularity} for $A$ w.r.t. $\text{zer}\; A$ is a modulus of regularity for $F$ w.r.t. $\text{zer}\; F$. We give next two instances when moduli of regularity for $A$ w.r.t. $\text{zer}\; A$ exist.

\begin{ex} \label{ex-zero-op}
\begin{itemize}
\item[(i)]  Let $X$ be a Banach space and  $X^* $ its dual. The normalized duality mapping $J : X \to 2^{X^*}$ is defined as 
\[J(x) =\{ j\in X^* : j(x) = \Vert x \Vert ^2 , \Vert j\Vert =\Vert x\Vert \}.\]

An operator $A:X\to  2^X $ is called {\it $\psi$-strongly accretive}, where $\psi :[0,\infty )\to [0,\infty )$ is a strictly increasing function with $\psi(0)=0$, if 
\begin{equation}\label{def-strongly-accretive}
\langle x^*-y^*, x-y\rangle_+ \geq \psi(\Vert x-y\Vert) \Vert x-y\Vert,
\end{equation}
for all $x,y\in X$, $x^* \in A(x)$, $y^* \in A(y)$, where $\langle  v, u \rangle_+  = \max \{ j(v) : j\in J(u)\}$. 
 
Assume that $\text{zer}\; A \ne \emptyset$ (hence it is a singleton) and let $x \in X, x \notin \text{zer}\; A$. Taking in \eqref{def-strongly-accretive} $y \in \text{zer}\; A$, we obtain 
\[\|x^*\| \ge \frac{\langle x^*,x-y\rangle_+}{\|x-y\|} \ge\psi(\Vert x-y\Vert),\]
for all $x^* \in A(x)$. Then it is clear that $A$ satisfies \eqref{eq-cond-oper} and $\phi : (0,\infty) \to (0,\infty)$, $\phi(\eps) := \psi(\eps)$, is a modulus of regularity for $A$ w.r.t. $\text{zer}\; A$. Furthermore, it is actually a modulus of uniqueness, a fact that was also observed in \cite[Remark 2]{KK15}. If $A$ is single-valued, then for any $\gamma > 0$, $\gamma\phi$ is a modulus of regularity for $\text{Id} - \gamma A$ w.r.t. $\text{Fix}\left(\text{Id} - \gamma A\right)$.

\item[(ii)] Metric subregularity has been extensively used in optimization in relation with stability problems and the linear local convergence of proximal point methods (see \cite{DR09,Le09}). An operator $A : X \rightarrow 2^Y$ is called {\it metrically subregular} at $z \in \text{zer}\; A$ for $O_Y$ if there exist $k,r > 0$ such that
$$\text{dist}(x,\text{zer}\;A) \le k \text{dist}(O_Y,A(x)),$$
for all $x \in \overline{B}(z,r)$. In this case, if $\text{zer}\; A$ is closed, \eqref{eq-cond-oper} holds and it immediately follows that $\phi : (0,\infty) \to (0,\infty)$, $\phi(\eps) := \eps/k$, is a modulus of regularity for $A$ w.r.t. $\text{zer}\;A$ and $\overline{B}(z,r)$. 
\end{itemize}
\end{ex}

\mbox{}

Recall that if $(X_1, d_1)$ and $(X_2, d_2)$ are metric spaces, a {\it modulus of uniform continuity} for a uniformly continuous mapping $T : X_1 \to X_2$ is a function $\rho:(0,\infty)\to (0,\infty)$ satisfying
$$\forall\eps>0 \, \forall x,y \in X_1 \, \left( d_1(x,y) <\rho(\eps) \rightarrow d_2(Tx,Ty) <\eps\right).$$

\begin{thm}\label{thm-equiv-mod}
Let $H$ be a Hilbert space and $f:H \to (-\infty,\infty]$ a proper, convex and lower semi-continuous function which attains its minimum. Take $z \in \argmin f$ and $r,r' > 0$. Consider the following statements:
\begin{itemize}
\item[1.] The function $f$ admits a modulus of regularity w.r.t. $\argmin f$ and $\overline{B}(z,r)$.
\item[2.] For $\gamma > 0$, the resolvent of $f$, $J_{\gamma \partial f}$, admits a modulus of regularity w.r.t. ${\rm Fix}\; J_{\gamma \partial f}$ and $\overline{B}(z,r)$.
\item[3.] The subdifferential of $f$, $\partial f$, admits a modulus of regularity w.r.t. ${\rm zer}\; \partial f$ and $\overline{B}(z,r')$.
\end{itemize}
Then 
\begin{itemize} 
\item[(i)] If $\left.f\right|_{\overline{B}(z,r+1)}$ is additionally uniformly continuous admitting a modulus of uniform continuity, then $1$ implies $2$ for all $\gamma > 0$.
\item[(ii)] If there exists $\gamma > 0$ such that $2$ holds, then $1$ is satisfied. Moreover, $3$ holds too if $r' < r$.
\item[(iii)] If $\partial f$ is single-valued, $r'=r$ and $\left.\left({\rm Id}+\gamma\partial f\right)\right|_{\overline{B}(z,r+1)}$, $\gamma > 0$, is uniformly continuous admitting a modulus of uniform continuity, then $3$ implies $2$.
\end{itemize}
\end{thm}
\begin{proof}
Recall first that 
$$\argmin \; f = \text{Fix} \; J_{\gamma \partial f} = \text{zer}\; \partial f,$$ 
for every $\gamma > 0$. 

(i) Let $\phi$ be a modulus of regularity for $f$ w.r.t. $\argmin f$ and $\overline{B}(z,r)$, and $\rho$ a modulus of uniform continuity for $\left.f\right|_{\overline{B}(z,r+1)}$. Fix $\gamma > 0$. Define $\phi^* : (0,\infty) \to (0,\infty)$ by 
$$\phi^*(\eps):=\min\left\{\rho\left(\frac{\phi(\eps)}{2}\right), \frac{\gamma\phi(\eps)}{2r}, 1\right\}.$$
To see that $\phi^*$ is a modulus of regularity for $J_{\gamma \partial f}$ w.r.t. $\text{Fix} \; J_{\gamma \partial f}$ and $\overline{B}(z,r)$, let $\eps > 0$ and $x \in \overline{B}(z,r)$. Assume $\|x - J_{\gamma \partial f}(x)\| < \phi^*(\eps)$. Then $J_{\gamma \partial f}(x) \in \overline{B}(z,r+1)$, so $f(x) - f(J_{\gamma \partial f}(x)) < \phi(\eps)/2$. At the same time, since $z \in \argmin \; f$, by \eqref{ineq-res},
\[f(J_{\gamma \partial f}(x)) - m \le \frac{1}{2\gamma}\left(\|z-x\|^2 - \|J_{\gamma \partial f}(x) - x\|^2 - \|J_{\gamma \partial f}(x) - z\|^2\right).\]
Because
\[\|z-x\|^2 - 2\|z-x\|\|J_{\gamma \partial f}(x) - x\| + \|J_{\gamma \partial f}(x) - x\|^2  = \left(\|z-x\| - \|J_{\gamma \partial f}(x) - x\|\right)^2 \le \|J_{\gamma \partial f}(x) - z\|^2,\]
we have
\begin{align*}
&\|z-x\|^2 - \|J_{\gamma \partial f}(x) - x\|^2 - \|J_{\gamma \partial f}(x) - z\|^2 \le 2 \|J_{\gamma \partial f}(x) - x\|\left(\|z-x\| -  \|J_{\gamma \partial f}(x) - x\|\right)\\
& \quad \le 2 \|J_{\gamma \partial f}(x) - x\|\|J_{\gamma \partial f}(x) - z\| = 2\|J_{\gamma \partial f}(x) - x\|\|J_{\gamma \partial f}(x) - J_{\gamma \partial f}(z)\| \\
& \quad \le 2\|J_{\gamma \partial f}(x) - x\|\|x - z\| < \gamma \phi(\eps).
\end{align*}
Thus, $f(J_{\gamma \partial f}(x)) - m < \phi(\eps)/2$, from where
\[f(x) - m = f(x) - f(J_{\gamma \partial f}(x)) + f(J_{\gamma \partial f}(x)) - m < \phi(\eps).\]
Consequently, $\text{dist}(x,\text{Fix}\; J_{\gamma \partial f}) < \eps$.

(ii) Let $\gamma > 0$ and $\phi$ be a modulus of regularity for $J_{\gamma \partial f} = \text{Prox}_{\gamma f}$ w.r.t. $\text{Fix} \; J_{\gamma \partial f}$ and $\overline{B}(z,r)$. 

We prove first that $\phi^* : (0,\infty) \to (0,\infty)$, 
\[\phi^*(\eps):=\frac{\phi(\eps)^2}{2\gamma},\]
is a modulus of regularity for $f$ w.r.t. $\argmin f$ and $\overline{B}(z,r)$.

To see this, let $\eps>0$ and $x \in \overline{B}(z,r)$ such that $f(x) - m < \phi^*(\eps)$.
Since
$$f(J_{\gamma \partial f}(x)) + \frac{1}{2\gamma}\|J_{\gamma \partial f}(x)-x\|^2 \le f(x),$$
we get
$$\|J_{\gamma \partial f}(x)-x\|^2 \le  2\gamma (f(x)- f(J_{\gamma \partial f}(x))) \le 2\gamma(f(x)-m).$$
Thus, $\|J_{\gamma \partial f}(x)-x\| < \phi(\eps)$, which yields $\text{dist}(x,\argmin f) < \eps$.

Define now $\phi^* : (0,\infty) \to (0,\infty)$ by 
$$\phi^*(\eps):=\frac{1}{\gamma}\min\left\{\phi\left(\frac{\eps}{2}\right),\frac{\eps}{2}, r-r'\right\}.$$
We show that $\phi^*$ is a modulus of regularity for $\partial f$ w.r.t. $\text{zer}\;\partial f$ and $\overline{B}(z,r')$.

Let $\eps>0$ and $x \in \overline{B}(z,r')$. Suppose $\text{dist}(O_Y,\partial f(x)) < \phi^*(\eps)$ and choose $y \in \partial f(x)$ such that $\|y\| < \phi^*(\eps)$. Then $x + \gamma y \in (\text{Id} + \gamma \partial f)(x)$, so $J_{\gamma\partial f}(x+\gamma y) = x$, 
\[\|x+\gamma y - z\| \le \|x - z\| + \gamma\|y\| \le r' + r - r' = r,\]
and
\[\|J_{\gamma\partial f}(x+\gamma y) - (x+\gamma y)\| = \gamma \|y\| < \phi\left(\frac{\eps}{2}\right).\] 
It follows that $\text{dist}(x+\gamma y,\text{zer}\; \partial f) < \eps/2$, hence
\[\text{dist}(x,\text{zer}\; \partial f) \le \text{dist}(x+\gamma y,\text{zer}\; \partial f) + \gamma\|y\| < \eps.\]

(iii) Note that in this case $\text{Id} + \gamma \partial f : H \to H$ is also bijective (see \cite[Chapter 23]{BC17}). Let $\rho$ be a modulus of uniform continuity for $\left.\left({\rm Id}+\gamma\partial f\right)\right|_{\overline{B}(z,r+1)}$ and $\phi$ a modulus of regularity for $\partial f$ w.r.t. $\text{zer}\;\partial f$ and $\overline{B}(z,r)$, and define $\phi^* : (0,\infty) \to (0,\infty)$ by 
$$\phi^*(\eps):=\min\left\{\rho(\gamma \phi(\eps)),1\right\}.$$
Let $\eps>0$ and $x \in \overline{B}(z,r)$ such that $\|J_{\gamma \partial f}(x)-x\| < \phi^*(\eps)$. Note that $J_{\gamma \partial f}(x) \in \overline{B}(z,r+1)$. Because
$$
\|\partial f(x)\| = \frac{1}{\gamma}\|\gamma \partial f(x)\| = \frac{1}{\gamma}\|(\text{Id} + \gamma \partial f)(J_{\gamma \partial f}(x)) - (\text{Id} + \gamma \partial f)(x)\| <\phi(\eps),
$$
it follows that $\text{dist}(x,\text{Fix}\; J_{\gamma \partial f}) < \eps$.
\end{proof}

\section{Rates of convergence} \label{section-rates}

\begin{thm}\label{thm-gen-rate}
Let $(X,d)$ be a metric space and $F : X \to \overline\R$ with ${\rm zer}\; F \ne \emptyset$. Suppose that $(x_n)$ is a sequence in $X$ which is Fej\'{e}r monotone w.r.t. ${\rm zer}\; F$, $b > 0$ is an upper bound on $d(x_0,z)$ for some $z \in {\rm zer}\; F$ and there exists $\alpha : (0,\infty) \to \N$ such that
\[\forall \eps > 0 \, \exists n \le \alpha(\eps) \, \left(|F(x_n)| < \eps\right).\]
If $\phi$ is a modulus of regularity for $F$ w.r.t. ${\rm zer}\; F$ and $\overline{B}(z,b)$, then $(x_n)$ is a Cauchy sequence with Cauchy modulus 
\begin{equation}\label{thm-gen-rate-eq1}
\forall \varepsilon>0 \, \forall n,\tilde{n}\ge \alpha(\phi(\eps/2)) \, \left(d(x_n,x_{\tilde{n}})<\varepsilon\right) 
\end{equation}
and 
\begin{equation}\label{thm-gen-rate-eq2} 
\forall \varepsilon>0 \, \forall n\ge \alpha(\phi(\eps)) \, \left( {\rm dist}(x_n,{\rm zer}\; F)<\eps\right). 
\end{equation}
Moreover, 
\begin{itemize}
\item[(i)] if $X$ is complete and ${\rm zer}\; F$ is closed, then $(x_n)$ converges to a zero of $F$ with a rate of convergence $\alpha(\phi(\varepsilon/2))$;
\item[(ii)] if there exists $\eps^* > 0$ such that
\begin{equation}\label{thm-gen-rate-eq3} 
\forall w \in \R, \, |w| < \eps^* \, \left( F^{-1}(w) \subseteq {\rm zer}\; F \cup \{x \in X : {\rm dist}(x,{\rm zer}\; F)\ge\eps^*\}\right), 
\end{equation}
then $x_n = z'$ for all $n \ge \alpha\left(\min\{\eps^*,\phi(\eps^*)\}\right)$, where $z' \in {\rm zer}\; F$.
\end{itemize}
\end{thm}
\begin{proof}
Let $\eps > 0$. Note that by Fej\'{e}r monotonicity, $(x_n) \subseteq \overline{B}(z,b)$. Since there exists $N \le \alpha(\phi(\varepsilon/2))$ such that $|F(x_N)| < \phi(\varepsilon/2)$, it follows that $\text{dist}(x_N,\text{zer}\; F) < \varepsilon/2$. Thus, $d(x_N,y) < \varepsilon/2$ for some $y \in \text{zer}\; F$. Since $(x_n)$ is Fej\'er monotone w.r.t. ${\rm zer}\; F$,  this implies that $d(x_n,y)\le d(x_N,y) < \varepsilon/2$ for all $n \ge \alpha(\phi(\eps/2))$, so \eqref{thm-gen-rate-eq1} and \eqref{thm-gen-rate-eq2} hold.

(i) If $X$ is complete, then $z' := \lim_{n\to\infty} x_n$ exists and, by the above Cauchy rate, we get that $d(x_m,z')\le \varepsilon$ for $m\geq\alpha(\phi(\eps/2))$. Hence,
\[\text{dist}(z',\text{zer}\; F)\le \text{dist}(x_m,\text{zer}\; F)+ d(z',x_m) < 3\eps/2.\] 
Since $\eps>0$ was arbitrary, we get $\text{dist}(z',\text{zer}\; F)=0$ which yields, if $\text{zer}\; F$ is closed, that $z'\in \text{zer}\; F$.

(ii) Let $N \le \alpha\left(\min\{\eps^*,\phi(\eps^*)\}\right)$ such that $|F(x_N)| < \min\{\eps^*,\phi(\eps^*)\}$. Taking $w := F(x_N)$ in \eqref{thm-gen-rate-eq3}, we obtain $x_N \in \text{zer}\; F \cup \{x \in X : {\rm dist}(x,\text{zer}\; F)\ge\eps^*\}$. However, as $|F(x_N)| < \phi(\eps^*)$, we have ${\rm dist}(x_N,\text{zer}\; F) < \eps^*$, and so $x_N = z'$ for some $z' \in \text{zer}\; F$. But then, by Fej\'er monotonicity, $d(x_{n},z') \le d(x_N,z') = 0$ for all $n \ge N$. Hence, $x_n = z'$ for all $n \ge \alpha\left(\min\{\eps^*,\phi(\eps^*)\}\right)$.
\end{proof}

\begin{re} \label{rmk-finite-conv}
If instead of \eqref{thm-gen-rate-eq3} one actually has the stronger condition
\begin{equation}\label{rmk-finite-conv-eq} 
\exists \eps^* > 0 \, \forall w \in \R, \, |w| < \eps^* \, \left( F^{-1}(w) \subseteq {\rm zer}\; F\right), 
\end{equation}
then one does not need to assume the existence of a modulus of regularity for $F$ in order to obtain the finite convergence of $(x_n)$. In this case the corresponding rate is $\alpha(\eps^*)$.
\end{re}

\begin{re}
In fact in order to obtain finite termination, condition \eqref{thm-gen-rate-eq3} does not need to hold for all points in $F^{-1}(w)$, but only for those also belonging to the set of values of $(x_n)$.
\end{re}

The following result is in some sense a converse of Theorem \ref{thm-gen-rate} for a particular situation.

\begin{pro}
Let $X$ be a metric space, $T : X \to X$ nonexpansive with $\text{Fix}\; T \ne \emptyset$, $z \in \text{Fix}\; T$ and $b>0$. If for every $x \in \overline{B}(z,b)$, $(T^n x)$ converges to a fixed point of $T$ with a common rate of convergence $\psi$, then $\phi(\eps) = \eps/(2\psi(\eps/2))$ is a modulus of regularity for $T$ w.r.t. $\text{Fix}\; T$ and $\overline{B}(z,b)$.
\end{pro}
\begin{proof}
Assume on the contrary that $\phi$ is not a modulus of regularity for $T$ w.r.t. $\text{Fix}\; T$ and $\overline{B}(z,b)$. Then there exist $\eps >0$ and $x \in \overline{B}(z,b)$ such that $d(x,Tx)<\phi(\eps)$ and $\text{dist}(x,\text{Fix}\; T) \ge \eps$. Denote $n = \psi(\eps/2)$. Then there exists $w \in \text{Fix}\; T$ such that $d(T^n x, w) < \eps/2$, which yields
\[\text{dist}(x,\text{Fix}\; T) \le d(x,w) \le d(w,T^n x) + \sum_{i = 0}^{n-1}d(T^i x, T^{i+1}x) < \frac{\eps}{2} + n d(x,Tx) < \frac{\eps}{2} + n\phi(\eps) = \eps,\]
a contradiction.
\end{proof}

In order to apply Theorem \ref{thm-gen-rate} to obtain rates of convergence for sequences $(x_n)$ converging to common fixed points of finitely 
many selfmappings $T_1,\ldots,T_n:X\to X$ of a metric space $(X,d)$ (which e.g. is the situation for iterative procedures to solve so-called  convex feasibility problems, see below) we first define the following notion.

\begin{de}[compare \cite{BauschkeBorwein96}]
Let $(X,d)$ be a metric space and $C_1,\ldots,C_m,K$ be subsets 
of $X$ with $C:=\bigcap^{m}_{i=1} C_i\not=\emptyset.$ We say that 
$C_1,\ldots,C_m$ are {\it (uniformly) metrically regular w.r.t. $K$} if 
\[ \forall \varepsilon>0\,\exists \delta>0\,\forall x\in K \ 
\left( \bigwedge^m_{i=1} \text{dist}(x,C_i)<\delta \rightarrow \text{dist}(x,C)<\varepsilon
\right). \]
We call a function $\rho: (0,\infty )\to  (0,\infty )$ producing such a $\delta=
\rho(\varepsilon)$ a {\it modulus of metric regularity} for 
$C_1,\ldots,C_m$ w.r.t. $K$.
\end{de}
\begin{ex}[\cite{Borwein14}]\label{ex2} From a result shown in 
{\rm \cite{Borwein14}}, the following is immediate: 
let $C_1,\ldots,C_m\subseteq \R^n$ be 
basic convex semi-algebraic sets given by 
\[ C_i:=\{ x\in \R^n | g_{i,j}(x)\le 0, \,j=1,\ldots ,m_i\},\] where 
$g_{i,j}$ are convex polynomials on $\R^n$ with degree at most $d\in\N.$ 
Then $\bigcap^m_{i=1} C_i\not=\emptyset$ and for any compact $K\subseteq \R^n$ there exists $c>0$ such that 
\[ \rho(\varepsilon):= (\varepsilon/c)^{\gamma}/m,\, \mbox{with} \, 
\gamma:=\min \left\{ \frac{(2d-1)^n+1}{2},B(n-1)d^n\right\},\] where $B(n):={{n}\choose{[n/2]}},$ is a modulus of metric regularity for 
$C_1,\ldots,C_m$ w.r.t. $K.$  
\end{ex}

As an easy consequence of Theorem \ref{thm-gen-rate} we obtain the next result.
\begin{co}\label{regularity-theorem}
Let $(X,d)$ be a complete metric space and $T_1,\ldots,T_m$ be selfmappings of $X$ with $C:=\bigcap^{m}_{i=1} \text{Fix}\; T_i$ nonempty and closed. Let $(x_n)$ be a sequence in $X$ which is Fej\'er monotone w.r.t. $C$ and assume that $b>0$ is an upper bound on $d(x_0,z)$ for some $z \in C$. Suppose that $\phi$ is a common modulus of regularity for $T_i$ w.r.t. $\text{Fix}\; T_i$ and $\overline{B}(z,b)$ for each $i=1,\ldots, m$.  If $\text{Fix}\; T_1,\ldots,\text{Fix}\; T_m$ are metrically regular w.r.t. $\overline{B}(z,b)$ with modulus $\rho$ and $(x_n)$ also has common approximate fixed points for $T_1,\ldots,T_m$ with $\alpha:(0,\infty)\to \N$ a common approximate fixed point bound, i.e.  
\[ \forall \varepsilon>0\,\exists n\le \alpha(\varepsilon) \ \left( 
\bigwedge^m_{i=1} d(x_n,T_i x_n)<\varepsilon\right),\]
then 
$(x_n)$ converges to a point in $C$ with a rate of convergence $\alpha(\phi(\rho(\varepsilon/2)))$. 
\end{co}
\begin{proof}
Denote $F : X \to \R$, $F(x) =  \max_{i = \overline{1,m}}d(x,T_i (x))$. Clearly, ${\rm zer}\; F =C$. Let $\eps > 0$ and $x \in \overline{B}(z,b)$ with $F(x) < \phi(\rho(\eps))$. Then $d(x,T_i x) < \phi(\rho(\eps))$ for all $i \in \{1, \ldots, m\}$, so $\text{dist}(x, \text{Fix}\;T_i) < \rho(\eps)$ for all $i \in \{1, \ldots, m\}$. Since the sets $\text{Fix}\; T_1,\ldots,\text{Fix}\; T_m$ are metrically regular w.r.t. $\overline{B}(z,b)$ with modulus $\rho$, we get $\text{dist}(x,{\rm zer}\; F) < \eps$. Thus, $\phi \circ\rho$ is a modulus of regularity for $F$ w.r.t. ${\rm zer}\; F$ and $\overline{B}(z,b)$. The result follows now from Theorem \ref{thm-gen-rate}.
\end{proof}

Recall that for retractions $T_i:X\to C_i(=\text{Fix}\; T_i)$, $\phi(\varepsilon)=\varepsilon$ is a modulus of regularity for $T_i$ w.r.t. $\text{Fix}\; T_i$. Consequently, we get the next result.
\begin{co}\label{cor-cyclic-retractions}
Let $C_1,\ldots,C_m\subseteq X$ be subsets of a complete metric space $(X,d)$ 
with $C:=\bigcap^m_{i=1} C_i$ nonempty and closed, and 
$T_i:X\to C_i,$ $i=1,\ldots,m$, be retractions. Then under the assumptions 
on $(x_n)$ and on the metric regularity of $C_1,\ldots,C_m$ from Corollary \ref{regularity-theorem}, one has $\alpha(\rho(\varepsilon/2))$ as a rate of convergence for $(x_n)$ to some point in $C$.
\end{co}

Note that Corollary \ref{cor-cyclic-retractions} applies in particular to so-called 
convex feasibility problems: having $X$ a complete CAT$(0)$ space 
(or even a CAT$(\kappa)$ space with $\kappa>0$ and an appropriate upper 
bound on its diameter) and $C_1,\ldots,C_m\subseteq X$ nonempty, closed and convex 
with $C:=\bigcap^m_{i=1} C_i\not=\emptyset$, the convex feasibility problem (CFP) consits of finding a point in $C$. Then we may apply Corollary \ref{cor-cyclic-retractions} with $T_i$ being the metric projection onto $C_i$.

\mbox{}

We study next the finite convergence of a sequence to a zero of a maximal monotone operator. To this end we assume condition \eqref{cor-finite-term-max-mon-eq1} which was considered by Rockafellar in \cite[Theorem 3]{R76} (see also \cite{Nev79}) to show that the proximal point algorithm terminates in finitely many iterations.
\begin{co}\label{cor-finite-term-max-mon}
Let $H$ be a Hilbert space and $A : H \to 2^H$ maximal monotone such that
\begin{equation}\label{cor-finite-term-max-mon-eq1}
\exists z \in H \; \exists \eps^* > 0 \;  \left(B(O,\eps^*) \subseteq A(z)\right).
\end{equation}
If $(x_n)$ is a sequence in $H$ that is a Fej\'{e}r monotone w.r.t. ${\rm zer}\; A$ and there exists $\alpha : (0,\infty) \to \N$ such that
\[\forall \eps > 0 \, \exists n \le \alpha(\eps) \, \left({\rm dist}(O,A(x_n)) < \eps\right),\]
then $\text{zer}\; A = \{z\}$ and $x_n = z$ for all $n \ge \alpha(\eps^*)$. 
\end{co}
\begin{proof}
Define $F : H \to \overline\R$, $F(x) := \text{dist}(O,A(x))$. We show that $F$ satisfies \eqref{rmk-finite-conv-eq}. For $w \in \R$, $0 \le w < \eps^*$ and $x \in F^{-1}(w)$, we have $\text{dist}(O,A(x)) = w < \eps^*$, so there exists $u \in A(x)$ such that $\|u\| < \eps^*$. Assume that $x \ne z$ and define
\[v_n := \frac{\eps^*}{1 + 1/n}\frac{x-z}{\|x-z\|}, \; n \in \N \setminus \{0\}.\] 
Note that $\|v_n\| < \eps^*$, so $v_n \in A(z)$. By the monotonicity of $A$, $\langle x-z, v_n\rangle \le \langle x-z, u\rangle$ for all $n \in \N \setminus \{0\}$, which yields $\eps^* \|x-z\| \le \|u\|\|x-z\|$. This is a contradiction, so $F^{-1}(w) = \{z\}$, which shows in particular that $\text{zer}\; A = \{z\}$. By Remark \ref{rmk-finite-conv}, $x_n = z$ for all $n \ge \alpha(\eps^*)$.
\end{proof}

We apply in the following the above results to different algorithms.

\subsection*{Picard iteration} 
Let $X$ be a complete metric space and $T:X\to X$ a quasi-nonexpansive mapping with $\text{Fix}\; T\not=\emptyset$. The {\it Picard iteration} generates starting from $x_0 \in X$ the sequence given by
\begin{equation}\label{Picard-algorithm} 
x_{n+1} := T x_n \quad \text{for any } n \in \N.
\end{equation}
It is well-known that $(x_n)$ is Fej\'{e}r monotone w.r.t. $\text{Fix}\; T$. Moreover, $\text{Fix}\; T$ is closed. Note also that if $T$ is nonexpansive, a function $\alpha : (0,\infty) \to \N$ such that 
\[\forall \eps > 0 \, \exists n \le \alpha(\eps) \, \left(d(x_n ,Tx_n ) < \eps\right)\]
is actually a rate of asymptotic regularity for $(x_n)$ as the sequence $(d(x_n,Tx_n))$ is nonincreasing.

Let $b > 0$ be an upper bound on $d(x_0,z)$ for some $z \in {\rm Fix}\; T$. By Fej\'{e}r monotonicity, $(x_n) \subseteq \overline{B}(z,b)$.
Considering $F:X \to \R, F(x)= d(x, Tx)$, if $\phi$ is a modulus of regularity for $F$ w.r.t. $\text{zer}\; F$ and $\overline{B}(z,b)$, and $\alpha$ is a rate of asymptotic regularity for $(x_n)$, then, applying Theorem \ref{thm-gen-rate}, we can deduce that $(x_n)$ converges to a fixed point of $T$ with a rate of convergence $\alpha(\phi(\varepsilon/2))$. In what follows we consider two problems where such $\alpha$ and $\phi$ can be computed explicitly.

\mbox{}

First we focus on the problem of minimizing the distance between two nonintersecting sets applying the alternating projection method. If $X$ is a complete CAT$(0)$ space and $U, V \subseteq X$ are nonempty, closed and convex with $U \cap V = \emptyset$, then one aims to find best approximation pairs $(u,v) \in U \times V$ such that $d(u,v) = \text{dist}(U,V)$. This problem was studied in \cite{Ban14,ALN15} (for results in Hilbert spaces, see, e.g., \cite{BB94,Kop12}). Denote $\rho := \text{dist}(U,V)$ and suppose that $S :=  \left\{(u,v) \in U \times V : d(u,v)=\rho\right\} \ne \emptyset$. Given $x_0 \in X$, we consider the sequence $(x_n)$ given by \eqref{Picard-algorithm}, where $T : X \to X$, $T := P_U \circ P_V$. Then:
\begin{itemize}
\item $T$ is nonexpansive.
\item If $(u,v) \in S$, then $u\in \text{Fix}\; T$, so $\text{Fix}\; T \ne \emptyset$. At the same time, if $u \in \text{Fix}\; T$, then $(u,P_V u) \in S$ (see \cite{ALN15}). We take $b$ an upper bound on $d(x_0,z)$ for some fixed $z \in \text{Fix}\; T$.
\item For any $x \in X$ and $u\in \text{Fix}\; T$,
\begin{equation}\label{min-pb-eq1}
d(Tx,P_V x)^2 \le \rho^2 + d(u,x)^2 - d(u,Tx)^2.
\end{equation}
To see this, apply first \eqref{eq-CAT0-proj} to get $d(Tx,P_V x)^2 + d(u,Tx)^2 \le d(u,P_V x)^2$ and $d(x,P_V x)^2 + d(P_V x, P_V u)^2 \le d(x,P_V u)^2$. Then
\begin{align*}
d(Tx,P_V x)^2 + d(u,Tx)^2 & \le d(u,P_V x)^2 + d(x,P_V u)^2 - d(x,P_V x)^2 - d(P_V x, P_V u)^2 \\
& \le d(u, P_V u)^2 + d(u,x)^2 \quad \text{by }\eqref{def-CAT0}.
\end{align*}
\item $(x_n)$ is asymptotically regular with a rate of asymptotic regularity (see \cite{ALN15})
\[\alpha_s(\eps):=\left[\frac{s}{\eps^2}\right] + 1,\]
where $s \ge d(T x_0, P_V x_0)^2$. By \eqref{min-pb-eq1}, we can take $s = \rho^2 + b^2$ and we obtain the following rate of asymptotic regularity
\[\alpha(\eps):=\left[\frac{\rho^2 + b^2}{\eps^2}\right] + 1.\]
\end{itemize}

We assume that the sets $U$ and $V$ are additionally {\it boundedly regular} which means that for any bounded set $K \subseteq X$ and any $\eps > 0$ there exists $\delta > 0$ such that for all $x \in K$ we have the following implication
\begin{equation}\label{min-pb-eq2}
\text{dist}(x,U) < \delta \wedge \text{dist}(x,V) < \rho + \delta  \; \Rightarrow \; \text{dist}(x,\text{Fix}\; T) < \eps.
\end{equation}
This notion is a natural analogue of the one given in \cite{BauBor93} in the setting of Hilbert spaces. Extensions of the concept of bounded regularity have been introduced in the recent paper \cite{CRZ} to analyze the speed of convergence of a sequence defined by a family of operators.

Let $\eps > 0$ and consider $K := \overline{B}(z,b)$. Since $U, V$ are boundedly regular, there exists $\delta = \delta(\eps) > 0$ such that \eqref{min-pb-eq2} holds for $x \in K$. We show next that $\phi : (0,\infty) \to (0,\infty)$,
\[\phi(\eps) := \frac{\rho \delta}{b + \rho},\]
is a modulus of regularity for $T$ w.r.t. $\text{Fix}\; T$ and $K$. To this end, let $x \in K$ such that $d(x,Tx) < \phi(\eps)$. Clearly, $\text{dist}(x,U) < \phi(\eps) < \delta$. From \eqref{min-pb-eq1}, it follows that $d(Tx,P_V x)^2 \le \rho^2 + 2d(x,Tx)d(z,x)$, so 
\[d(Tx,P_V x)^2 \le \rho^2 + 2b\phi(\eps).\]
Now,
\[\text{dist}(x,V) \le d(x, P_V x) \le d(x,Tx) + d(Tx,P_V x) < \phi(\eps) + \sqrt{\rho^2 + 2b\phi(\eps)}.\]
An easy computation shows that $\text{dist}(x,V) < \rho + \delta$. Thus, $\text{dist}(x,\text{Fix}\; T) < \eps$.

\mbox{}

Another algorithm which fits into the scheme \eqref{Picard-algorithm} is the gradient descent method employed to find minimizers of convex functions. Let $H$ be a Hilbert space and $f:H\to \R$ convex, Fr\'{e}chet differentiable on $H$ and such that its gradient $\nabla f$ is $L$-Lipschitz. Suppose that $\argmin\; f \ne \emptyset$. Given $x_0 \in H$, the gradient descent method with constant step size $1/L$ generates the sequence $(x_n)$ given by \eqref{Picard-algorithm}, where $T : H \to H$, $T := \text{Id} -\frac{1}{L}\nabla f$. Then the following facts are known: 
\begin{itemize}
\item $\argmin\; f \subseteq \text{Fix}\; T$, so $\text{Fix}\; T \ne \emptyset$. We take $b$ to be an upper bound on $d(x_0,z)$ for some $z \in \text{Fix}\; T$.
\item $T$ is firmly nonexpansive by the Baillon-Haddad theorem \cite{BH77}.
\item $(x_n)$  is asymptotically regular with a rate of asymptotic regularity (see \cite{AriLeuLop14})
\[\alpha(\eps):=\left[\frac{32(b+1)^2}{\eps^2}\right].\]
\end{itemize}

In addition, suppose that $\nabla f$ is $\psi$-strongly accretive. Then $\phi : (0,\infty)  \to (0,\infty)$, $\phi(\eps) = \psi(\eps)/L$ is a modulus of regularity for $T$ w.r.t. $\text{Fix}\; T$ (see Example \ref{ex-zero-op}.(i)).

\mbox{}

In \cite{Kohlenbach(SNE)}, it is shown that the Picard iteration $(x_n)$ of 
the composition $T:=T_m\circ\ldots \circ T_1$ of finitely many metric projections $T_i:=P_{C_i}$ of a complete 
CAT$(\kappa)$ space $X$ (with $\kappa > 0$ and diameter less than $\pi/(2\sqrt{\kappa})$) onto 
closed and convex sets $C_i \subseteq X$, $i=1,\ldots,m$, with 
$C:=\bigcap^m_{i=1} C_i \not=\emptyset$ is asymptotically regular and has 
common approximate fixed points. Moreover, an explicit common approximate 
fixed point bound (in the sense of Corollary 
\ref{regularity-theorem}) is given (follows from \cite[Corollaries 4.17, 4.5]{Kohlenbach(SNE)} and the Lipschitz continuity of $T_i$). Since metric 
projections in CAT$(\kappa)$ spaces are quasi-nonexpansive (and so is 
$T$ since the fixed points of $T$ are precisely 
the common fixed points of $T_1,\ldots,T_m$, see \cite{Kohlenbach(SNE)})  one gets 
the Fej\'er monotonicity w.r.t. $C$ of the sequence $(x_n)$. A related algorithm to approach the CFP in this setting is the cyclic projection method. In this case, $(x_n)$ is not a Picard iteration, but is generated by $x_{n+1} := T_{\bar n} x_n$, where $n \in \N$, $x_0 \in X$ and $T_{\bar n} := T_{n (\text{mod } m) + 1}$. The sequence $(x_n)$ is again Fej\'{e}r monotone w.r.t. $C$ and an explicit common approximate fixed point bound when $X$ is a CAT$(0)$ space can be obtained from \cite[Theorem 3.2, Remark 3.1]{ALN15}. Thus, our general results on rates of convergence are applicable in these situations (see Corollary \ref{cor-cyclic-retractions} and the comment below it).

\mbox{}

We finish this subsection with the following observations.
\begin{re} \label{re-noncomp}
There exists a computable firmly nonexpansive mapping $T:[0,1]\to [0,1]$ such that the computable Picard iteration $x_n:=T^n 0$ is convergent and does not have a computable rate of convergence.
\end{re}
\begin{proof}
We use a construction from \cite{Neumann}: let $(a_n)$ be a 
so-called Specker sequence, i.e. a computable 
nondecreasing sequence of rational numbers in $[0,1]$ without a computable 
limit (which exists by \cite{Specker(49)}). Define 
\[f_n:[0,1]\to [0,1], \ f_n(x):=\max\{ x,a_n\} \] 
and put 
\[ T(x):=\frac{1}{2} (x+f(x)), \ \mbox{where} \ f(x):=\sum^{\infty}_{n=0} 
2^{-n-1} f_n(x). \]
Then $f:[0,1]\to [0,1]$ is nonexpansive with $\text{Fix}\; f=[a,1],$ where 
$a:=\lim_{n\to\infty} a_n,$ and so $T$ is firmly nonexpansive 
with $\text{Fix}\;T=[a,1]$. Since $x_n\le x_{n+1} \le a$, $(x_n)$ converges to 
a fixed point of $T$ which must be $a$. If $(x_n)$ had a computable 
rate of convergence, then $a$ would be computable, which is a contradiction.
\end{proof}

\begin{re}\label{fn-noncomp}
Since $T$ is firmly nonexpansive, the sequence $(x_n)$ defined in Remark \ref{re-noncomp} is Fej\'{e}r monotone w.r.t. $\text{Fix}\; T$ and asymptotically regular with an explicit rate of asymptotic regularity. Thus, by Theorem \ref{thm-gen-rate} (applied to $F(x):=|x-Tx|$), $T$ has no computable modulus of regularity w.r.t. $\text{Fix}\; T$.
\end{re}

\subsection*{Mann iteration}

Let $X$ be a uniquely geodesic space and $T : X \to X$ with $\text{Fix}\; T \ne \emptyset$. The {\it Mann iteration} associated to $T$ starting from $x_0 \in X$ is defined by 
\begin{equation}\label{def-Mann-ne}
x_{n+1}:=(1-\lambda_n)x_n + \lambda_nTx_n \quad \text{for any } n \in \N,
\end{equation}
where the coefficients $\lambda_n$ are in $[0,1]$. 

Suppose next that $X$ is a $\CAT(0)$ space, $T$ is nonexpansive and $b > 0$ is an upper bound on $d(x_0,z)$ for some $z \in {\rm Fix}\; T$. Note that $(x_n)$ is Fej\'{e}r monotone w.r.t. $\text{Fix}\; T$. If additionally $(\lambda_n)$ satisfies $\sum_{n=0}^\infty \lambda_n(1-\lambda_n) = \infty$ with rate of divergence $\theta$, then it was proved in \cite{Leu07} that $(x_n)$ is asymptotically regular with the following rate of asymptotic regularity
\[\alpha(\varepsilon ):=\theta\left( \left\lceil\frac{4(b +1)^2}{\eps^2}\right\rceil\right).\]

In the setting of Hilbert spaces, the Mann algorithm has been used in combination with splitting methods to solve problems that can be abstracted into finding a zero of the sum of two maximal monotone operators. Let $A, B:H\to2^H $ be two maximal monotone operators with $\text{zer}(A+B) \ne \emptyset$ and let $\gamma > 0$. The Douglas-Rachford algorithm is the Mann algorithm with $T := R_{\gamma A} R_{\gamma B}$. Note that in this case, as mentioned in Section \ref{section-prelim}, $T$ is a nonexpansive mapping defined on $H$. Since, by \cite[Proposition 26.1]{BC17}, $\text{zer}(A+B)= J_{\gamma B}(\text{Fix}\;  T)$, we have $\text{Fix}\;  T \ne \emptyset$. If $\phi$ is a modulus of regularity for $T$ w.r.t. $\text{Fix}\; T$ and $\overline{B}(z,b)$, then, by Theorem \ref{thm-gen-rate}, the sequence $(x_n)$ converges to a fixed point of $T$ with rate of convergence $\alpha(\phi(\eps/2))$. 

In particular, if $C_1$, $C_2$ and $T$ are as in Example \ref{ex-mwu-fix}.(iv), then we obtain a rate of convergence for $(x_n)$ to a fixed point of $T$ whose projection onto $C_1$ lies in $C_1\cap C_2 = \text{zer}(N_{C_2}+N_{C_1})$. Consequently, the sequence $(P_{C_1} x_n)$ converges to a point in $C_1\cap C_2$ with the same rate of convergence.

\mbox{}

The CFP can also be solved using a Mann-type iteration studied by Crombez \cite{Cro91,Cro92} which was analyzed quantitatively in \cite{KhanKohlenbach14} (see also \cite{Rei83,Nic13}). Let $H$ be a Hilbert 
space and $C_1,\ldots,C_m\subseteq H$ be closed and convex subsets 
with $C:=\bigcap^m_{i=1} C_i\not=\emptyset$. For $1\le i\le m$, 
let $P_{C_i}:H \to C_i$ be metric projections, $T_i:=\text{Id}+\lambda_i(P_{C_i}-\text{Id})$ with $0<\lambda_i\le 2, \, \lambda_1<2$, and put $T:=\sum^m_{i=1} a_iT_i$, where $a_1,\ldots,a_m \in (0,1)$ with $\sum_{i=1}^m a_i =1$. As shown in \cite{KhanKohlenbach14}, $T$ can be written as $T=a\text{Id}+(1-a)S$ for suitable $a \in (0,1)$ and nonexpansive $S : C \to C$ which satisfies $\text{Fix}\; S=C$.  Let $x_0\in H$ and $b \ge \| x_0-z\|$ for some $z\in C$. 
The sequence  $x_n:=T^{n}x_0$ is Fej\'er monotone w.r.t. $\text{Fix}\; S$ 
since it is the Mann iteration associated to $S$ with constant 
coefficient $a$. Moreover, the sequences $(\|x_n - P_{C_i}x_n\|)_n$, $i = 1, \ldots, m$, are asymptotically regular with a rate of asymptotic regularity $\alpha^*$ which is quartic in $1/\eps$ (see \cite[Corollary 4.3.(i)] {KhanKohlenbach14} where the exact expression of $\alpha^*$ is given). In particular, $\alpha^*$ is a common approximate fixed point bound for $P_{C_1},\ldots,P_{C_m}$. We can now apply Corollary \ref{cor-cyclic-retractions} to obtain that the sequence $(x_n)$ converges to a point in $C$ with a rate of convergence $\alpha^*(\rho(\eps/2))$ whenever 
$C_1,\ldots,C_m$ are metrically regular w.r.t. 
$\overline{B}(z,b)$ with modulus $\rho.$ This, in particular, applies to 
$H=\R^n$ and the situation of Example \ref{ex2} with 
the modulus of metric regularity $\rho$ given there.

\subsection*{Proximal point algorithm}
Let $H$ be a Hilbert space and $A : H \rightarrow 2^H$ a maximal monotone operator with $\text{zer}\; A \ne \emptyset$. Note that $\text{zer}\; A$ is closed. Given $x_0 \in H$ and a sequence of positive numbers $(\gamma_n)$, the {\it proximal point algorithm} (PPA) generates the sequence defined by
\begin{equation}\label{eq-ppa}
x_{n+1} := J_{\gamma_n A} x_n \quad \text{for any } n \in \N.
\end{equation}
It is well-known that $(x_n)$ is Fejer monotone w.r.t. $\text{zer}\; A$. Denoting $F : H \to \overline{\R}$, $F(x) =  \text{dist}(O,A(x))$ and $u_n = \frac{x_n - x_{n+1}}{\gamma_n}$, we have $F(x_{n+1}) \le \|u_n\|$ for all $n \in \N$. Take $b>0$ an upper bound of $\|x_0 - z\|$ for some $z \in \text{zer}\; A$. If $\sum_{n=0}^\infty \gamma_n^2 = \infty$ with a rate of divergence $\theta$, then, by \cite[Lemma 8.3.(ii)]{KLA}, $\theta\left(\left\lceil \frac{b^2}{\eps^2}\right\rceil\right)$ is a rate of convergence of $(\|u_n\|)$ towards $0$. Therefore,
\[\forall \eps > 0 \, \forall n \ge \theta\left(\left\lceil \frac{b^2}{\eps^2}\right\rceil\right) + 1 \, \left(F(x_n) \le \eps\right)\]
and so $\forall \eps > 0 \, \left(F(x_{\alpha(\eps)}) < \eps\right)$, where $\alpha(\eps) := \theta\left(\left\lceil \frac{2b^2}{\eps^2}\right\rceil\right) + 1$. Thus, if $\phi$ is a modulus of regularity for $A$ w.r.t. $\text{zer}\; A$ and $\overline{B}(z,b)$, then, by Theorem \ref{thm-gen-rate}, $(x_n)$ converges to some $z' \in \text{zer}\; A$ with a rate of convergence $\alpha(\phi(\eps/2,b))$. In addition, if \eqref{cor-finite-term-max-mon-eq1} holds, then we can apply Corollary \ref{cor-finite-term-max-mon} to get $x_n = z'$ for all $n \ge \alpha(\eps^*)$. This gives a quantitative version of a special form of \cite[Theorem 3]{R76}.

\mbox{}

The PPA has been extensively applied as a method to localize a minimizer of a convex function. Let $f:H \to (-\infty,\infty]$ be proper, convex and lower semi-continuous and $S = \argmin\; f \ne \emptyset$. In this case, the sequence $(x_n)$ is given by \eqref{eq-ppa} for $A = \partial f$ and we take $z$ and $b$ as above. 

The following additional conditions allow us to give an explicit modulus of regularity for $\partial f$ w.r.t. $\text{zer}\; \partial f$ and $\overline{B}(z,b)$. If $S$ is a set of $\psi$-boundedly global weak sharp minima for $f$, then $\phi^* : (0,\infty) \to (0,\infty)$, $\phi^*(\eps)=\psi_C(\eps)$, where $C := \overline{B}(z,b+1)$, is a modulus of regularity for $f$ w.r.t. $S$ and $\overline{B}(z,b+1)$ (see Example \ref{ex-min-pb}.(ii)). Suppose that $\left.f\right|_{\overline{B}(z,b+2)}$ is uniformly continuous with a modulus of uniform continuity $\rho$. Applying Theorem \ref{thm-equiv-mod} and using the expressions of the moduli computed in its proof, one obtains that $\phi : (0,\infty) \to (0,\infty)$, 
\[\phi(\eps):=\min\left\{\rho\left(\frac{\psi_C(\eps/2)}{2}\right),\frac{\psi_C(\eps/2)}{2(b+1)},\frac{\eps}{2},1\right\},\]
is a modulus of regularity for $\partial f$ w.r.t. $\text{zer}\; \partial f$ and $\overline{B}(z,b)$.

We finish with a particular situation when we obtain finite convergence of $(x_n)$. Suppose that $S$ is a set of $\psi$-global weak sharp minima for $f$ with $\psi(\eps) = k \,\eps$, where $k > 0$. We prove that $F$ defined as before satisfies \eqref{rmk-finite-conv-eq}. Applying \cite[Theorem 2, Lemma 5]{Fer91}, it follows that there exists $\eps^* > 0$ such that if $x \in H$, $u \in \partial f(x)$ with $\|u\| \le \eps^*$, then $x \in S$. For $w \in \R$, $0 \le w < \eps^*$ and $x \in F^{-1}(w)$, we have $\text{dist}(O,\partial f(x)) = w < \eps^*$, so there exists $u \in \partial f(x)$ such that $\|u\| < \eps^*$. Therefore, $x \in S$, which yields $F^{-1}(w) \subseteq \text{zer}\; F$. By Remark \ref{rmk-finite-conv}, $x_n = z'$ for all $n \ge \alpha(\eps^*)$, where $z' \in S$ and $\alpha$ is given above. This gives a quantitative version of the main result in \cite{Fer91}.

%%%%%%%%%%%%%%%%%%%%%%%%%%%%%%%%%%%%%%%%%%%%%%%%%%%%%%%%%%%%%%%%%%%%%%%%%%%

\end{document}